\documentclass[12pt]{article}
\usepackage{amssymb}
\usepackage{amsfonts}
\usepackage{times}
\usepackage{mathptmx}
\usepackage{amsmath}
\usepackage[usenames]{color}
\usepackage{mathrsfs}
\usepackage{amsfonts}
\usepackage{amssymb,amsmath}
\usepackage{CJK}
\usepackage{cite}
\usepackage{cases}
\usepackage{amsthm}

\def\OP#1{\raisebox{-7pt}{$\stackrel{\dis\oplus}{\ssc #1}$}}

\def\cl{\centerline}

\def\a{\alpha}

\def\b{\beta}
\def\vs{\vspace*}

\def\Z{\mathbb{Z}}

\def\C{\mathbb{C}}

\def\OP{\oplus}

\numberwithin{equation}{section}
\newtheorem{theo}{Theorem}[section]
\newtheorem{defi}[theo]{Definition}

\newtheorem{lemm}[theo]{Lemma}
\newtheorem{prop}[theo]{Proposition}

\begin{document}
\begin{center}
{\bf\large Loop $W(a,b)$ Lie conformal algebra}
\footnote {Supported by NSF grant no. 11371278, 11431010.

$^{\,\dag}$ Corresponding author:  H.~Wu.}
\end{center}

\cl{Guangzhe Fan$^{\,\P,\S}$, Henan Wu$^{\,\dag,}$~and Bo Yu$^{\,\P,\ddag}$}

\cl{\small $^{\,\P}$Department of Mathematics, Tongji University, Shanghai 200092, China}

\cl{\small $^{\,\dag}$ School of Mathematical Sciences, Shanxi University, Taiyuan 030006, China}

\cl{\small $^{\,\S}$yzfanguangzhe@126.com}

\cl{\small $^{\,\dag}$wuhenan@sxu.edu.cn}

\cl{\small $^{\,\ddag}$15221506237@163.com }

\vs{8pt}

{\small\footnotesize
\parskip .005 truein
\baselineskip 3pt \lineskip 3pt
\noindent{{\bf Abstract:} Fix $a,b\in\C$, let $LW(a,b)$ be the loop $W(a,b)$ Lie algebra over $\C$ with basis $\{L_{\a,i},I_{\b,j} \mid \a,\b,i,j\in\Z\}$ and relations $[L_{\a,i},L_{\b,j}]=(\a-\b)L_{\a+\b,i+j}, [L_{\a,i},I_{\b,j}]=-(a+b\a+\b)I_{\a+\b,i+j},[I_{\a,i},I_{\b,j}]=0$, where $\a,\b,i,j\in\Z$. In this paper, a formal distribution Lie algebra of $LW(a,b)$ is constructed. Then the associated conformal algebra $CLW(a,b)$ is studied, where $CLW(a,b)$ has a $\C[\partial]$-basis $\{L_i,I_j\,|\,i,j\in\Z\}$ with $\lambda$-brackets $[L_i\, {}_\lambda \, L_j]=(\partial+2\lambda) L_{i+j}, [L_i\, {}_\lambda \, I_j]=(\partial+(1-b)\lambda) I_{i+j}$ and $[I_i\, {}_\lambda \, I_j]=0$. In particular, we determine the conformal derivations and rank one conformal modules of this conformal algebra. Finally, we study the central extensions and extensions of conformal modules. \vs{5pt}

\parskip .001 truein\baselineskip 6pt \lineskip 6pt

\section{Introduction}	
The notion of Lie conformal algebras was introduced by V. G. Kac as a formal language describing the singular part of the operator product expansion in conformal field theory. It is useful to research infinite dimensional Lie algebras satisfying the locality property. Recently, the structure theory and representation theory of some Lie conformal algebras have been extensively studied in \cite{BKV,FK,S,SY1,SY2}. For example, finite irreducible conformal modules over the Virasoro conformal algebra were determined in \cite{CK}. In addition, the conformal derivations and conformal modules of some infinite rank Lie conformal algebras were studied in \cite{FSW,GXY,WCY}. However, there is little about the central extensions and modules extensions of infinite rank Lie conformal algebras. In this paper, we will investigate the structure theory and representation theory of an infinite rank Lie conformal algebra, called loop $W(a,b)$ Lie conformal algebra, such as rank one conformal modules, central extensions and modules extensions. We believe this article would play an energetic role on the study of infinite rank Lie conformal algebras.

It is well known that the Lie algebra $W(a,b)$ is an important infinite dimensional Lie algebra, whose theory plays a crucial role in many areas of mathematics and physics. Fix $a,b\in\C$, then $W(a,b)=\bigoplus_{\a\in\Z}({\C L_{\a}}\bigoplus{\C I_{\a}})$ could be a Lie algebra equipped with the following brackets:
\begin{eqnarray*}
\aligned
&[L_{\a},L_{\b}]=(\a-\b) L_{\a+\b}, \ \ 
&[L_{\a},I_{\b}]=-(a+b\a+\b) I_{\a+\b}, \ \ \ \ 
&[I_{\a},I_{\b}]=0,
\endaligned
\end{eqnarray*}
for any $\a,\b\in\Z$. Obviously, we know that $W(a,b)\simeq W\ltimes I(a,b)$, where $W$ is the well-known centerless Virasoro and $I(a,b)$ is the tensor density module of $W$.  The structure theory and representation theory of $W(a,b)$ were developed in \cite{GJP,GLZ,SXY}. Furthermore, $W(a,b)$ turns out to be some famous Lie algebras for special values of $a, b$. For example, $W(0,0)$ is the well-known twisted Heisenberg-Virasoro algebra whose structure theory and Harish-Chandra modules were investigated in \cite{LJ,SJ}. In addition, $W(0,-1)$ is the W-algebra W(2,2) which was first introduced and studied in \cite{ZD}.

The loop $W(a,b)$ Lie algebra $LW(a,b)$ is defined to be the tensor product of $W(a,b)$ and the Laurent polynomial algebra $\C[t,t^{-1}]$, which is a Lie algebra with basis $\{L_{\a,i},I_{\b,j}\,|\,\a,\,\b,$ $i,j\in\Z\}$ and Lie brackets given by
\begin{eqnarray*}
\aligned
&[L_{\a,i},L_{\b,j}]=(\a-\b) L_{\a+\b,i+j},\\
&[L_{\a,i},I_{\b,j}]=-(a+b\a+\b) I_{\a+\b,i+j},\\
&[I_{\a,i},I_{\b,j}]=0,
\endaligned
\end{eqnarray*}
for any $\a,\b,i,j\in\Z$.
The subalgebra spanned by
$\{L_{\a,i}\,|\,\a,i\in\Z\}$
is actually isomorphic to the centerless loop-Virasoro algebra,  whose structure theory was studied in \cite{TZ}. Furthermore, simple Harish-Chandra modules, intermediate series modules, and Verma modules of the loop-Virasoro algebra were investigated in \cite{GLZ}.

Infinite rank Lie conformal algebras are important ingredients of Lie conformal algebras. In this article, we would like to
study an infinite rank Lie conformal algebra, namely, the loop $W(a,b)$ Lie conformal algebra $CLW(a,b)$ (cf.~\eqref{1.1})~for some $a,b\in\C$.
We construct the Lie conformal algebra $CLW(a,b)$ by $LW(a,b)$ in Section $3$.
As one can see, it is a Lie conformal algebra with $\C[\partial]$-basis $\{L_i,I_j\,|\,i,\,j\in\Z\}$ and $\lambda$-brackets
\begin{eqnarray}\label{1.1}
\aligned
&[L_i\, {}_\lambda \, L_j]=(\partial+2\lambda) L_{i+j}, \ \
[L_i\, {}_\lambda \, I_j]=(\partial+(1-b)\lambda) I_{i+j},\ \
[I_i\, {}_\lambda \, I_j]=0.
\endaligned
\end{eqnarray}
We should mention that $CLW(a,b)$ contains many important conformal subalgebras.
For example, the conformal subalgebra $CVir=\C[\partial]L_{0}$ is isomorphic to the well-known Virasoro conformal algebra and
the conformal subalgebra $CW=\bigoplus_{i\in\Z}\C[\partial]L_i$ is isomorphic to the loop Virasoro Lie conformal algebra studied in \cite{WCY}. Furthermore, we know that $CLW(0,0)$ is the loop Heisenberg-Virasoro Lie conformal algebra which was studied in \cite{FSW}. In addition, \cite{XY} constructed the $W(a,b)$ Lie conformal algebra for some $a,b$ and its conformal module of rank one. Therefore, we can apply some results of these Lie conformal algebras.

Let us briefly describe the structure of the article. In Section $2$, we introduce some basic definitions and previous results of Lie conformal algebras. In Section $3$, on the one hand, we look for the condition such that $F$ is a $\C[\partial]$-module, on the other hand, we construct a formal distribution Lie algebra $(LW(a,b), F)$ for a suitable family $F$ of pairwise local formal distributions. In Sections $4$ and $5$, we study conformal derivations and free nontrivial rank one conformal modules of $CLW(a,b)$. Finally, the central extensions and modules extensions of $CLW(a,b)$ are classified in Sections $6$ and $7$ respectively.

Throughout this paper, we denote by $\C,\,\C^*,\, \Z,\, \Z^{+}$ the sets of complex numbers, nonzero complex numbers, integers, nonnegative integers respectively. We assume that the indices $i,j,k\in \Z$, unless otherwise stated.

\section{Preliminaries}
In this section, we summarize some definitions related to formal distribution Lie algebras and Lie conformal algebras in \cite{DK,K1,K3}.

\begin{prop}\label{202}
Let $g$ be a Lie algebra. If (a(z),b(w)) is a local pair of $g$-valued formal distributions, the Fourier coefficients satisfy the following commutation relation on $g$:
\begin{equation*}
[a_{(m)},b_{(n)}]=\sum_{j\in\Z^{+}} \big(\begin{array}{c}
                                      m \\
                                      j
                                    \end{array}\big)
(a_{(j)}b)_{(m+n-j)},
\end{equation*}
where\begin{equation*}
a_{(j)}b=a(w)_{(j)}b(w)={\rm Res}_z (z-w)^{j}[a(z),b(w)],
\end{equation*}
is called the j-product of $a(w)$ and $b(w)$.
\end{prop}

\begin{defi}\label{203}\rm
Let $g$ be a Lie algebra. The $\lambda$-bracket of two $g$-valued formal distributions is defined by the $\C$-bilinear map
\begin{equation*}\label{204}
[\cdot_{\lambda}\cdot]:g[[w,w^{-1}]]\otimes g[[w,w^{-1}]]\longrightarrow g[[w,w^{-1}]][[\lambda]]
\end{equation*}
with
\begin{equation}\label{205}
[a(w)\,{}_\lambda\, b(w)]=F^\lambda_{z,w}[a(z),b(w)],
\end{equation}
where $F^\lambda_{z,w}a(z,w)={\rm Res}_z e^{\lambda(z-w)}a(z,w)$ is the formal Fourier transform.
\end{defi}

Furthermore, one readily shows that the $\lambda$-bracket is related to the $j$-product as follows:
\begin{equation}\label{206}
[a\,{}_\lambda\, b]=\sum_{j\in\Z^{+}}\frac{\lambda^j}{j!}(a_{(j)}b).
\end{equation}

\begin{defi}\label{250}\rm
A Lie conformal algebra is a $\C[\partial]$-module $A$ endowed with a linear map $A\otimes A\rightarrow A[\lambda]$, $a\otimes b\rightarrow [a{}\, _\lambda \, b]$, called $\lambda$-bracket, where $\lambda$ is an indeterminate and $A[\lambda]=\C[\lambda]\otimes A$, subject to the following three axioms:
\begin{equation}\label{251}
\aligned
&Conformal~~sesquilinearity:~~~~[\partial a\,{}_\lambda \,b]=-\lambda[a\,{}_\lambda\, b],\ \ \ \
[a\,{}_\lambda \,\partial b]=(\partial+\lambda)[a\,{}_\lambda\, b];\\
&Skew~~symmetry:~~~~[a\, {}_\lambda\, b]=-[b\,{}_{-\lambda-\partial}\,a];\\
&Jacobi~~identity:~~~~[a\,{}_\lambda\,[b\,{}_\mu\, c]]=[[a\,{}_\lambda\, b]\,{}_{\lambda+\mu}\, c]+[b\,{}_\mu\,[a\,{}_\lambda \,c]].
\endaligned
\end{equation}
\end{defi}

\begin{defi}\label{252}\rm
A conformal module $M$ over a Lie conformal algebra $A$ is a $\C[\partial]$-module endowed with a $\lambda$-action $A\otimes M\rightarrow M[\lambda]$ such that
\begin{equation*}
\aligned
&(\partial a)\,{}_\lambda\, v=-\lambda a\,{}_\lambda\, v,\ \ \ \ \ a{}\,{}_\lambda\, (\partial v)=(\partial+\lambda)a\,{}_\lambda\, v;\\
&a\,{}_\lambda\, (b{}\,_\mu\, v)-b\,{}_\mu\,(a\,{}_\lambda\, v)=[a\,{}_\lambda\, b]\,{}_{\lambda+\mu}\, v.
\endaligned
\end{equation*}
Furthermore, if $M$ is a free module of rank one over $A$, we call $M$ a rank one conformal module over $A$.
\end{defi}
\begin{defi}\rm
A Lie conformal algebra $A$ is {\it $\Z$-graded} if $A=\oplus_{i\in \Z}A_i$, where each $A_i$ is a $\C[\partial]$-submodule
and $[A_i\,{}_\lambda\, A_j]\subset A_{i+j}[\lambda]$ for any $i,j\in \Z$.
\end{defi}

\begin{defi}\rm
Let $V$ and $W$ be two $\C[\partial]$-modules. A conformal linear map from $V$ to $W$ is a $\C$-linear map $\phi_\lambda:V\longrightarrow \C[\partial][\lambda]\otimes_{\C[\partial]}W$ such that
\begin{equation*}\label{253}
\phi_\lambda(\partial v)=(\partial+\lambda)\phi_\lambda(v),\ \ \mbox{ \ for $v\in V$.}
\end{equation*}
\end{defi}
Denote the space of conformal linear maps between $\C[\partial]$-modules $V$ and $W$ by $Chom(V,W)$.

\begin{defi}\rm
Let $A$ be a Lie conformal algebra. A conformal linear map $D_\lambda:A\longrightarrow A[\lambda]$ is called a conformal derivation if
\begin{equation*}\label{254}
D_\lambda([a\,{}_\mu \,b])=[(D_\lambda a)\,{}_{\lambda+\mu} \,b]+[a\,{}_\mu \,(D_\lambda b)],\ \ \mbox{ \ for all $a,b,c\in A$.}
\end{equation*}
\end{defi}
It can be easily verified that for any $x\in A$, the map ${\rm ad}_x$, defined by $({\rm ad}_x)_\lambda y= [x\, {}_\lambda\, y]$ for $y\in A$, is a conformal derivation of $A$. All conformal derivations of this kind are called conformal inner derivations.
Denote by ${\rm CDer\,}(A)$
and ${\rm CInn\,}(A)$ the vector spaces of all conformal derivations and conformal inner derivations of $A$ respectively.

\section{The Lie conformal algebra $CLW(a,b)$}
In this section, we would like to start with the Lie algebra $LW(a,b)$ to construct the Lie conformal algebra $CLW(a,b)$ via formal distribution Lie algebra.
Let $F$ be a vector space spanned by $\{L_i(z),\,I_j(z)\,|\,i,j\in\Z\}$, where $L_i(z)=\sum_{\a\in\Z} L_{\a,i}z^{-\a-2}$ and  $I_j(z)=\sum_{\b\in\Z} I_{\b,j}z^{-\b-x}$ for any $i,j\in\Z$ and some $x\in\Z$.

According to the result¡¡of the $W(a,b)$ Lie conformal algebra studied in \cite{XY}, we can get Proposition~3.1 similarly.
\begin{prop}\label{3.1}
$F$ is a $\C[\partial]$-module if and only if $x=a-b+1$.
\end{prop}



By Proposition \ref{3.1}, we can obtain that if and only if $a-b\in\Z$, there exists the Lie conformal algebra $CLW(a,b)$. Moreover, we have propositions as follows.

\begin{prop}\label{3.2}If $a-b\in\Z$, then we have $I_j(z)=\sum_{\b\in\Z} I_{\b,j}z^{-\b-a+b-1}$. Also, we have
\begin{eqnarray*}&&
[L_i(z),L_j(w)]=(\partial_w L_{i+j}(w))\delta(z,w)+2L_{i+j}(w)\partial_w\delta(z,w),
\\&&[L_i(z),I_j(w)]=(\partial_w I_{i+j}(w))\delta(z,w)+(1-b)I_{i+j}(w)\partial_w\delta(z,w),
\\&&[I_i(z),I_j(w)]=0.
\end{eqnarray*}
\end{prop}

\begin{prop}\label{3.3} If $a-b\in\Z$, then we have
\begin{equation}\label{3.4}
\aligned
&[L_i\,{}_\lambda\, L_j]=(\partial+2\lambda)L_{i+j},
\ \ 
[L_i\,{}_\lambda\, I_j]=(\partial+(1-b)\lambda)I_{i+j},
\ \ 
[I_i\,{}_\lambda\, I_j]=0.
\endaligned
\end{equation}
\end{prop}

\begin{prop}\label{3.5} Let $CLW(a,b)$ be a free $\C[\partial]$-module with $\C[\partial]$-basis $\{L_i, I_i \mid i,j\in\Z\}$ and $a-b\in\Z$. Then $CLW(a,b)$ is a Lie conformal algebra with $\lambda$-brackets defined as in Proposition~3.3.
\end{prop}

For convenience, we simply denote $R(b)=CLW(a,b)$. Note that $R(b)$ is a $\Z$-graded Lie conformal algebra in the sense $R(b)=\oplus_{i\in\Z} \ (R(b))_{i}$, where
$(R(b))_{i}=\C[\partial]{L_i}\oplus \C[\partial]{I_i}$. Obviously, $R(0)$ is the loop Heisenberg-Virasoro Lie conformal algebra studied in \cite{FSW}.

\section{Conformal derivations of $CLW(a,b)$}
Suppose $D\in {\rm CDer\,} (R(b))$. Define $D^i(L_j)=\pi_{i+j} D(L_j),\,D^i(I_j)=\pi_{i+j} D(I_j)$ for any $j\in\Z$, where in general $\pi_{i}$ is the natural projection from $$\C[\lambda]\otimes R(b)\cong \OP_{j\in\Z}\C[\partial,\lambda]L_j\oplus\OP_{j\in\Z}\C[\partial,\lambda]I_j,$$ onto $\C[\partial,\lambda]{L_{i}}\oplus\C[\partial,\lambda]{I_{i}}$.
Then $D^i$ is a conformal derivation and $D=\sum_{i\in\Z} D^i$ in the sense that for any $x\in R(b)$ only finitely many $D^i_\lambda(x)\neq0$.
Let ${({\rm CDer\,}(R(b)))}^c$ be the space of conformal derivations of degree $c$, i.e.,
 $${({\rm CDer\,}(R(b)))}^c=\{D\in {\rm CDer\,}(R(b))\,|\, D_\lambda((R(b))_i)\subset (R(b))_{i+c}[\lambda]\}.$$

Firstly, by \cite{FSW}, we have the following theorem.
\begin{theo}
${\rm CDer\,}(R(0))={\rm CInn\,}(R(0))\oplus \C^\infty$.
\end{theo}

For readers' convenience, the definition of $\C^\infty$ is given. Denote
$$\C^\infty=\{\vec{a}=(a_c)_{c\in\Z}\,|\,a_c\in\C\mbox{ and }a_c=0\ \mbox{for\ all\ but\ finitely\ many}\ c's\}.$$
For each $\vec{a}\in \C^\infty$, we define $D_{\vec{a}}\, {}_\lambda\, (L_i)=\sum a_c I_{i+c}$ and $ D_{\vec{a}}\, {}_\lambda\,(I_i)=0$ for all $i\in\Z$.
Then $D_{\vec{a}}\in {\rm CDer\,}(R(0))$. It  can be easily verified that $D_{\vec{a}}\in {\rm CInn\,}(R(0))$ implies $D_{\vec{a}}=0$.
For simplicity, we denote by $\C^\infty$ the space of such conformal derivations.

Methods similar to those in the article \cite{FSW} are used to discuss nonzero parameter $b$. As a result, we obtain the following theorem.
\begin{theo}If $b\in\C^*$, we have
${\rm CDer\,}(R(b))={\rm CInn\,}(R(b))$.
\end{theo}

Now we get the main result of this section.
\begin{theo}
${\rm CDer\,}(R(b))={\rm CInn\,}(R(b))\oplus \delta_{b,0}\C^\infty$.
\end{theo}

\section{Rank one conformal modules over $CLW(a,b)$}
Now suppose $M$ is a free conformal module of rank one over $R(b)$.
We write $M=\C[\partial]v$ and assume $L_i\,{}_\lambda\, v=f_i(\partial,\lambda)v$, $I_j\,{}_\lambda\, v=g_j(\partial,\lambda)v$, where $f_i(\partial,\lambda),g_j(\partial,\lambda)\in\C[\partial,\lambda]$. We will compute the coefficients $f_i(\partial,\lambda),g_j(\partial,\lambda)$.

Firstly, according to \cite{WCY}, we have Lemma 5.1.
\begin{lemm}\label{500}
There exist $\Delta,\alpha,c\in\C$ such that $f_i(\partial,\lambda)=c^i(\partial+\Delta\lambda+\alpha)$~for any~$i\in\Z$.
\end{lemm}

Using the similar method of \cite{FSW}, we conclude the following lemma.
\begin{lemm}\label{502}
$g_i(\partial,\lambda)=\delta_{b,0}d c^i$ ~where~ $c,d\in\C$, $i\in\Z$.
\end{lemm}

From above discussions, we obtain the main result of this section.
\begin{theo}\label{507}
A nontrivial free conformal module of rank one over $R(b)$ is isomorphic to $M(\Delta,\alpha,c,d)$ for some $\Delta,\alpha,d\in\C, c\in\C^*$, where $M(\Delta,\alpha,c,d)=\C[\partial]v$ and $\lambda$-actions are given by
\begin{equation*}
L_i\, {}_\lambda \,v=c^i(\partial+\Delta\lambda+\alpha)v,\ \ I_i\, {}_\lambda \,v=\delta_{b,0}d c^i v.
\end{equation*}
Furthermore, $M(\Delta,\alpha,c,d)$ is irreducible if and only if $\Delta\neq0$.
\end{theo}

\section{The central extensions of $CLW(a,b)$}
In this section, we shall consider the central extensions of $R(b)$.

An extension of a Lie conformal algebra $R$ by an abelian Lie conformal algebra $G$ is a short exact sequence of Lie conformal algebra
\begin{equation*}
0\rightarrow G \rightarrow \widehat{R} \rightarrow R \rightarrow 0.
\end{equation*}
That is to say, $\widehat{R}$ is called an extension of $R$ by $G$. This extension is said to be central if
\begin{equation*}
G\subset Z(\widehat{R})=\{x\in\widehat{R}\,|\, [x\,_\lambda \, y]=0, \forall~y\in \widehat{R}\},~\partial G=0.
\end{equation*}
Let $\widehat{R}$ be a central extension of $R$ by a one-dimensional center $\C\mathfrak{c}$. This means that $\widehat{R}\cong R\bigoplus \C\mathfrak{c}$ as vector spaces, and
$$[a\,_\lambda\, b]_{\widehat{R}}=[a\,_\lambda\, b]_{R}+\phi_{\lambda}(a, b)\mathfrak{c},$$
where $\phi_\lambda: R\times R \rightarrow \C[\lambda]$ is a bilinear map. It follows from the axioms of Lie conformal algebra that $\phi_\lambda$ satisfy:
\begin{equation}\label{601}
\aligned
&\phi_{\lambda}(a, b)=- \phi_{-\partial-\lambda}(b, a);\\
&\phi_{\lambda}(\partial a, b)=-\lambda \phi_{\lambda}(a, b)=-\phi_{\lambda}(a,\partial b);\\
&\phi_{\lambda+\mu}([a\,_\lambda \, b], c)=\phi_{\lambda}(a, [b\,_\mu\, c])-\phi_{\mu}(b, [a\,_\lambda\, c]);
\endaligned
\end{equation}
for all $a,b,c\in R$. The map $\phi_\lambda$ satisfying (\ref{601}) is called a $2$-cocycle of $R$.

Now we will compute the central extension $\widehat{R(b)}$ of $R(b)$ by a one-dimensional center $\C\mathfrak{c}$, i.e., $\widehat{R(b)}=R(b)\oplus\C\mathfrak{c}$, and the $\lambda$-brackets (\ref{3.4}) are replaced by
\begin{eqnarray*}&&
[L_i\,{}_\lambda\, L_j]=(\partial+2\lambda)L_{i+j}+A_{\lambda}(L_{i},L_{j})\mathfrak{c},
\\&&[L_i\,{}_\lambda\, I_j]=(\partial+(1-b)\lambda)I_{i+j}+B_{\lambda}(L_{i},I_{j})\mathfrak{c},
\\&&[I_i\,{}_\lambda\, I_j]=C_{\lambda}(I_{i},I_{j})\mathfrak{c},
\end{eqnarray*}
where $A_{\lambda},B_{\lambda},C_{\lambda}:R(b)\otimes R(b)\rightarrow\C[\lambda]$ are bilinear maps.

In the following, our main work is to determine $A_{\lambda}(L_{i},L_{j}),B_{\lambda}(L_{i},I_{j})$ and $C_{\lambda}(I_{i},I_{j})$.

Firstly, according to the conclusion of \cite{H}, we have the following equality
\begin{equation}\label{620}
A_{\lambda}(L_{i},L_{j})=A(i+j)\lambda+A'(i+j)\lambda^3,
\end{equation}
where $A$ and $A'$ are complex functions.

Next we apply the Jacobi identity to $(L_{i},L_{j},I_{k})$, and one has
\begin{eqnarray}\label{621}
&&(\lambda-\mu)B_{\lambda+\mu}(L_{i+j}, I_{k})=(\lambda+(1-b)\mu)B_{\lambda}(L_{i}, I_{j+k})-(\mu+(1-b)\lambda)B_{\mu}(L_{j}, I_{i+k}).
\end{eqnarray}
We denote $B_{\lambda}(L_{i},I_{j})=\sum_{m=0}^{n}b_{m}(L_{i},I_{j})\lambda^{m}\in \C[\lambda]$ with $b_{n}(L_{i},I_{j})\neq0$.
Then we get\begin{eqnarray}\label{622}
&&(\lambda-\mu)\sum_{m=0}^{n}b_{m}(L_{i+j},I_{k})(\lambda+\mu)^{m}=(\lambda+(1-b)\mu)\sum_{m=0}^{n}b_{m}(L_{i},I_{j+k})\lambda^{m}-(\mu+(1-b)\lambda)\sum_{m=0}^{n}b_{m}(L_{j},I_{i+k})\mu^{m}.
\end{eqnarray}

By using the similar methods of computing $A_{\lambda}(L_{i},L_{j})$, one can deduce that
\begin{eqnarray*}&&
b_{0}(L_{j},I_{k})=\delta_{b,1}B(i+j),~~~b_{1}(L_{j},I_{k})=B'(i+j),
\\&&b_{2}(L_{j},I_{k})=\delta_{b,0}B''(i+j), ~~~b_{3}(L_{j},I_{k})=\delta_{b,-1}B'''(i+j)
\end{eqnarray*}
where $B$, $B'$, $B''$ and $B'''$ are complex functions.

It is easily verified that $b_{m}(L_{i},I_{j})=0$ for $m>3$.

Therefore, we have
\begin{equation}\label{623}
B_{\lambda}(L_{i},I_{j})=\delta_{b,1}B(i+j)+B'(i+j)\lambda+\delta_{b,0}B''(i+j)\lambda^2+\delta_{b,-1}B'''(i+j)\lambda^3.
\end{equation}

Applying the Jacobi identity to $(L_{i},I_{j},I_{k})$ and comparing the coefficients of $\lambda^i \mu^j$, we show that
\begin{equation}\label{624}
C_{\lambda}(I_{i},I_{j})=\delta_{2b,1}C(i+j)+\delta_{b,0}C'(i+j)\lambda.
\end{equation}
where $C$ and $C'$ are complex functions.

Finally, from above discussions, we obtain the main result of this section.
\begin{theo}\label{6000}
The one-dimensional central extension $\widehat{R(b)}$ of $R(b)$ has the following form:
\begin{eqnarray*}&&
[L_i\,{}_\lambda\, L_j]=(\partial+2\lambda)L_{i+j}+(A(i+j)\lambda+A'(i+j)\lambda^3)\mathfrak{c},
\\&&[L_i\,{}_\lambda\, I_j]=(\partial+(1-b)\lambda)I_{i+j}+(\delta_{b,1}B(i+j)+B'(i+j)\lambda+\delta_{b,0}B''(i+j)\lambda^2+\delta_{b,-1}B'''(i+j)\lambda^3)\mathfrak{c},
\\&&[I_i\,{}_\lambda\, I_j]=(\delta_{2b,1}C(i+j)+\delta_{b,0}C'(i+j)\lambda)\mathfrak{c},
\end{eqnarray*}
where $A$, $A'$, $B$, $B'$, $B''$, $B'''$, $C$  and $C'$ are complex functions.
\end{theo}

In particular, we have the one-dimensional central extension of loop Heisenberg-Virasoro Lie conformal algebra studied in \cite{FSW}
\begin{eqnarray*}&&
[L_i\,{}_\lambda\, L_j]=(\partial+2\lambda)L_{i+j}+(A(i+j)\lambda+A'(i+j)\lambda^3)\mathfrak{c},
\\&&[L_i\,{}_\lambda\, I_j]=(\partial+\lambda)I_{i+j}+(B'(i+j)\lambda+B''(i+j)\lambda^2)\mathfrak{c},
\\&&[I_i\,{}_\lambda\, I_j]=C'(i+j)\lambda\mathfrak{c}.
\end{eqnarray*}

\section{The extensions of conformal modules}
Above all we review some definitions about extensions of conformal modules.

For two conformal modules $V$ and $W$ over a Lie conformal algebra $R$, an extension $E$ of $W$ by $V$ is a module over $R$ which satisfies an exact sequence
\begin{equation*}
0\rightarrow V\rightarrow E\rightarrow W\rightarrow0,
\end{equation*}


Clearly, an extension can be regarded as the direct sum of vector spaces $E=V\oplus W$, where $V$ is a submodule of $E$, while for $w$ in $W$ we can obtain
\begin{equation*}
a_\cdot w=aw+\phi_{a}(w),\mbox{ \ $a\in R$,}
\end{equation*}
where $\phi_{a}:W\rightarrow V$ is a linear map satisfying the cocycle condition:~$\phi_{[a,b]}(w)=\phi_{a}(bw)+a\phi_{b}(w)-\phi_{b}(aw)-b\phi_{a}(w)$, $b\in R$. The set of these cocycles forms a vector space over $\C$. Cocycles corresponding to the trivial extension are called trivial cocycles. The dimension of the quotient space is called the dimension of the space of extensions of $W$ by $V$. The quotient space is denoted by ${\rm Ext}(W, V)$.

By Theorem~5.3, we have a nontrivial free conformal module $M(\Delta,\alpha,c,d)=\C[\partial]v$ of rank one over $R(b)$ with $\lambda$-actions given by
\begin{equation*}
L_i\, {}_\lambda \,v=c^i(\partial+\Delta\lambda+\alpha)v,\ \ I_i\, {}_\lambda \,v=\delta_{b,0}d c^i v.
\end{equation*}

We first consider extensions of $R(b)$-modules of the form
\begin{equation*}
0\longrightarrow \C_\b \longrightarrow E \longrightarrow M(\Delta, \alpha, c, d)\longrightarrow 0,
\end{equation*}
where $\C_\b=\C v_\b$ denotes the 1-dimensional vector space over $\C$ on which we have an action of $R(b)$ defined by
\begin{equation*}
L_{\lambda}v_\b=I_{\lambda}v_\b=0,\ \ \partial v_\b= \beta v_\b.
\end{equation*}
Then we have $E=\C[\partial]v\oplus \C_\b$ and $\lambda$-action defined by
\begin{equation}
\aligned
&L_i\,_\lambda\, v_\b=I_i\,_\lambda\, v_\b=0,\\
&L_i\,_\lambda\, v= c^i(\partial+\Delta \lambda +\a)v+f_i(\lambda)v_\b,\\
&I_i\,_\lambda\, v= \delta_{b,0}d c^iv+g_i(\lambda)v_\b,\\
\endaligned
\end{equation}
where $f_i(\lambda),g_i(\lambda)\in\C[\lambda]$.

According to \cite{W}, we have the following two lemmas. Since \cite{W} is a doctor dissertation, for readers' convenience, we supply an outline of the proof.
Firstly, from the definition of the conformal module, we can get the following equality
\begin{equation}\label{7000005}
(\lambda-\mu)f_{i+j}(\lambda+\mu)=c^j(\a+\b+\lambda+\Delta\mu)f_{i}(\lambda)-c^i(\a+\b+\Delta\lambda+\mu)f_{j}(\mu).
\end{equation}

Letting $i=j=\mu=0$ and $i=\mu=0$ in (\ref{7000005}), it follows that Lemma~7.1 holds.

At the same time, using the similar method of \cite{CKW}, we conclude Lemma~7.2.

\begin{lemm}\label{701}
If $\a+\b\neq 0$, then $f_i(\lambda)=0$.
\end{lemm}

\begin{lemm}\label{703}
If $\a+\b=0$, one has
\begin{equation*}
f_i(\lambda)=
\begin{cases}
ic^{i-1}e, &\ \text{if~}\  \Delta=-1,\\[4pt]
c^{i}\lambda^2e, &\ \text{if~}\  \Delta=1,\\[4pt]
c^{i}\lambda^3e, &\ \text{if~}\  \Delta=2,\\[4pt]
0,&\  \text{otherwise}.
\end{cases}
\end{equation*}
where $e\in\C$.
\end{lemm}

Next we would like to investigate $g_i(\lambda)$.

\begin{lemm}\label{704}
For any~$i, j\in\Z$,
\begin{equation}\label{705}
(b\lambda+\mu)g_{i+j}(\lambda+\mu)=c^i(\a+\b+\Delta\lambda+\mu)g_{j}(\mu)-\delta_{b,0}d c^jf_{i}(\lambda).
\end{equation}
\end{lemm}
\begin{proof}
A direct computation shows that
\begin{eqnarray*}\label{c2}
&&[L_i\,{}_\lambda\, I_j]\,{}_{\lambda+\mu} v=((\partial+(1-b)\lambda)I_{i+j})_{\lambda+\mu} v\nonumber\\
&=&-(b\lambda+\mu)g_{i+j}(\lambda+\mu) v_\b\ \  ({\rm mod}\ \C[\partial, \lambda, \mu]v),
\end{eqnarray*}
\begin{eqnarray*}\label{c2}
&&[L_i\,{}_\lambda\, I_j]\,{}_{\lambda+\mu} v=L_i\,{}_\lambda\,(I_j\, {}_\mu v)-I_j\,{}_\mu\, (L_i \,{}_\lambda v)\nonumber\\
&=&(\delta_{b,0}d c^jf_{i}(\lambda)-c^i(\a+\b+\Delta\lambda+\mu)g_{j}(\mu)) v_\b\ \  ({\rm mod}\ \C[\partial, \lambda, \mu]v).
\end{eqnarray*}

Thus, the lemma holds.
\end{proof}

\begin{lemm}\label{706}
If $\a+\b\neq 0$, then
$g_{i}(\lambda)=0$.
\end{lemm}
\begin{proof}
Letting $\lambda=i=0$ in (\ref{705}), we obtain
\begin{equation}\label{707}
\mu g_{j}(\mu)=(\a+\b+\mu)g_{j}(\mu).
\end{equation}
Due to $\a+\b\neq 0$, we have $g_{j}(\mu)=0$.

This completes the proof.
\end{proof}

\begin{lemm}\label{708}
If $\a+\b=0$ and $b\neq0$, then
\begin{equation*}
g_i(\lambda)=
\begin{cases}
r c^{i}, &\ \text{if~}\  \Delta=b,\\[4pt]
0,&\  \text{otherwise}.
\end{cases}
\end{equation*}
where $r\in\C$.

\end{lemm}
\begin{proof}
If $\a+\b=0$ and $b\neq0$,
then (\ref{705}) can be rewritten by
\begin{equation}\label{709}
(b\lambda+\mu)g_{i+j}(\lambda+\mu)=c^i(\Delta\lambda+\mu)g_{j}(\mu).
\end{equation}
Letting $i=j=0$ in (\ref{709}), one has
\begin{equation}\label{707}
(b\lambda+\mu)g_{0}(\lambda+\mu)=(\Delta\lambda+\mu)g_{0}(\mu).
\end{equation}
If $\Delta\neq b$, it can be easily seen that $g_{0}(\mu)=0$.

Letting $\mu=j=0$ in (\ref{709}), it follows that
\begin{equation}\label{710}
b g_{i}(\lambda)=c^i \Delta g_{0}(0).
\end{equation}
Therefore, we have $g_{i}(\lambda)=0$.

If $\Delta=b$ in (\ref{710}), we have $g_{i}(\lambda)=rc^i$ with $r=g_{0}(0)$.

So we have completed the proof.
\end{proof}

\begin{lemm}\label{711}
If $\a+\b=b=0$, we obtain \begin{equation*}
g_i(\lambda)=
\begin{cases}
c^{i}\lambda, &\ \text{if~}\  \Delta=1,~d=0,\\[4pt]
0,&\  \text{otherwise}.
\end{cases}
\end{equation*}
\end{lemm}
\begin{proof}
If $b=0$, according to\begin{eqnarray*}\label{c2}
&&0=[I_i\,{}_\lambda\, I_j]\,{}_{\lambda+\mu} v\nonumber\\
&=&[I_i\,{}_\lambda\, I_j]\,{}_{\lambda+\mu} v=I_i\,{}_\lambda\,(I_j\, {}_\mu v)-I_j\,{}_\mu\, (I_i \,{}_\lambda v)\nonumber\\
&=&d^2 c^{i+j}(g_{i}(\lambda)-g_{j}(\mu)) v_\b\ \  ({\rm mod}\ \C[\partial, \lambda, \mu]v),
\end{eqnarray*}
then we have $d^2 c^{i+j}(g_{i}(\lambda)-g_{j}(\mu))=0$.

If $c,d\in\C^*$, then $g_{i}(\lambda)$ is a constant for any $i\in\Z$. Thanks to (\ref{705}), we at once deduce $g_{i}(\lambda)=0$.

If $\a+\b=b=d=0$,
then (\ref{705}) becomes
\begin{equation}\label{712}
\mu g_{i+j}(\lambda+\mu)=c^i(\Delta\lambda+\mu)g_{j}(\mu).
\end{equation}
Setting $i=j=0$ in (\ref{712}), we obtain
\begin{equation}\label{713}
\mu g_{0}(\lambda+\mu)=(\Delta\lambda+\mu)g_{0}(\mu).
\end{equation}

If $\Delta\neq 1$, it shows that $g_{i}(\lambda)=0$.
If $\Delta=1$, we get $g_{i}(\lambda)=Ac^{i}\lambda$, where $A\in\C$.

Therefore, we have proved the lemma.
\end{proof}

Moreover, if $\a+\b=b=0$, $c,d\in\C^*$, according to (\ref{705}), then we have $f_{i}(\lambda)=0$.

Finally, we can easily obtain the following two theorems.
\begin{theo}
If $b, c\in\C^*$, one deduce
\begin{equation*}
{\rm dim}\ {\rm Ext}\,(M(\Delta, \alpha, c, d), \C_\b)=
\begin{cases}
1+\delta_{b,-1}+\delta_{b,1}+\delta_{b,2}, &\text{if}\ \a+\b=0,\ \Delta=b,\\
0, &\text{otherwise}.
\end{cases}
\end{equation*}
\end{theo}

\begin{theo}
If $b=0, c\in\C^*$, one follows
\begin{equation*}
{\rm dim}\ {\rm Ext}\,(M(\Delta, \alpha, c, d), \C_\b)=
\begin{cases}
1, &\ \text{if}\  \a+\b=d=0, \ \Delta=-1~or~2,\\
2, &\ \text{if}\  \a+\b=d=0,\ \Delta=1,\\
0, &\  \text{otherwise}.
\end{cases}
\end{equation*}
\end{theo}

Next we will compute ${\rm Ext}\,(\C_\b, M(\Delta, \alpha, c, d))$, i.e., the extension of $\C_\b$ by $M(\Delta, \alpha, c, d)$:
\begin{equation*}
0\longrightarrow M(\Delta, \alpha, c, d) \longrightarrow E \longrightarrow  \C_\b\longrightarrow 0,
\end{equation*}
we have $E=\C[\partial]v\oplus\C_\b$, where $\C[\partial]v\cong M(\Delta, \alpha, c, d)$ is a conformal submodule, and
\begin{equation*}
\aligned
&\partial v_\b=\b v_\b+\rho(\partial)v,\\
&L_i\,{}_\lambda\,v_\b=h_{i}(\partial, \lambda)v,\\
&I_i\,{}_\lambda\,v_\b=l_{i}(\partial, \lambda)v,
\endaligned
\end{equation*}
where $\rho(\partial)\in\C[\partial], h_{i}(\partial, \lambda), l_{i}(\partial, \lambda)\in \C[\partial, \lambda]$.

According to \cite{W}, we have the following lemma.
\begin{lemm}\label{751}
\begin{equation*}
h_{i}(\partial, \lambda)=
\begin{cases}
s c^i, &\ \text{if}\  \a+\b=0,~\Delta=1,\\
0, &\  \text{otherwise},
\end{cases}
\end{equation*}
where $s=\rho(\partial)\in\C$.
\end{lemm}

Note from $I_i\,{}_\lambda\,(\partial v_\b)=(\partial+\lambda)l_{i}(\partial, \lambda)v$, we get
\begin{equation}\label{752}
(\partial+\lambda)l_{i}(\partial, \lambda)=\b l_{i}(\partial, \lambda)+\rho(\partial+\lambda)\delta_{b,0}d c^i.
\end{equation}
If $b\neq0$, by comparing the highest degree of $\partial$ in (\ref{752}), we have $l_{i}(\partial, \lambda)=0$.

\begin{lemm}
If $b=0$, one shows that $l_{i}(\partial, \lambda)=0$.
\end{lemm}
\begin{proof}
Firstly, using\begin{eqnarray*}\label{c2}
&&0=[I_i\,{}_\lambda\, I_j]\,{}_{\lambda+\mu} v_\b=I_i\,{}_\lambda\,(I_j\, {}_\mu v_\b)-I_j\,{}_\mu\, (I_i \,{}_\lambda v_\b)\nonumber\\
&=&l_{j}(\partial+\lambda, \mu)l_{i}(\partial, \lambda)-l_{i}(\partial+\mu, \lambda)l_{j}(\partial, \mu),
\end{eqnarray*}
we conclude that \begin{equation}\label{75200}
l_{j}(\partial+\lambda, \mu)l_{i}(\partial, \lambda)=l_{i}(\partial+\mu, \lambda)l_{j}(\partial, \mu).
\end{equation}

By comparing the highest degree of $\lambda$ in both sides of (\ref{75200}), we conclude that
\begin{equation}\label{753}
l_{i}(\partial, \lambda)=l_{i}(\lambda).
\end{equation}

By\begin{eqnarray*}\label{c2}
&&[L_i\,{}_\lambda\, I_j]\,{}_{\lambda+\mu} v_\b=-(b\lambda+\mu)l_{i+j}(\partial, \lambda+\mu) v.
\end{eqnarray*}
\begin{eqnarray*}\label{c2}
&&[L_i\,{}_\lambda\, I_j]\,{}_{\lambda+\mu} v_\b=L_i\,{}_\lambda\,(I_j\, {}_\mu v_\b)-I_j\,{}_\mu\, (L_i \,{}_\lambda v_\b)\\
&=&c^i(\partial+\Delta\lambda+\a)l_{j}(\partial+\lambda, \mu)-\delta_{b,0}d c^i h_{i}(\partial+\mu, \lambda) v,
\end{eqnarray*}
we obtain
\begin{equation}\label{754}
-(b\lambda+\mu)l_{i+j}(\partial, \lambda+\mu)=c^i(\partial+\Delta\lambda+\a)l_{j}(\partial+\lambda, \mu)-\delta_{b,0}d c^i h_{i}(\partial+\mu, \lambda).
\end{equation}
If $b=0$, according to (\ref{753}), then (\ref{754}) can be rephrased by
\begin{equation}\label{755}
-\mu l_{i+j}(\lambda+\mu)=c^i(\partial+\Delta\lambda+\a)l_{j}(\mu)-d c^i h_{i}(\lambda).
\end{equation}
Comparing the degree of $\partial$ in (\ref{755}),
we obtain $l_{i}(\partial, \lambda)=l_{i}(\lambda)=0$.

Hence, the proof is finished.
\end{proof}

Moreover, if $\a+\b=b=0$, $c,d\in\C^*$, according to (\ref{755}), then we have $h_{i}(\lambda)=0$.

Finally, we get the following theorems.

\begin{theo}
If $b,c\in\C^*$, we have
\begin{equation*}
{\rm dim}\ {\rm Ext}\,(\C_\b, M(\Delta, \alpha, c, d))=
\begin{cases}
1, &\ \text{if}\  \a+\b=0,\ \Delta=1,\\
0, &\  \text{otherwise}.
\end{cases}
\end{equation*}
\end{theo}

\begin{theo}
If $b=0,c\in\C^*$, one has
\begin{equation*}
{\rm dim}\ {\rm Ext}\,(\C_\b, M(\Delta, \alpha, c, d))=
\begin{cases}
1, &\ \text{if}\  \a+\b=d=0,\ \Delta=1,\\
0, &\  \text{otherwise}.
\end{cases}
\end{equation*}
\end{theo}

\small

\end{document}